\documentclass[preprint]{elsarticle}
\UseRawInputEncoding
\usepackage[normalem]{ulem}
\usepackage{makecell}
\usepackage{booktabs}
\usepackage{array}
\usepackage{multirow}
\usepackage[suffix=]{epstopdf}
\usepackage[toc,page]{appendix}
\usepackage{listings}
\usepackage{enumitem}

\usepackage[numbered,framed]{matlab-prettifier}
\usepackage{textcomp}

\usepackage{mathtools}
\usepackage{color}

\usepackage{amssymb,amsmath}
\usepackage{amsfonts}
\usepackage{graphics}
\usepackage{amsthm}

\definecolor{mygreen}{RGB}{28,172,0} 
\definecolor{mylilas}{RGB}{170,55,241}

\newtheorem{remark}{Remark}
\newtheorem{theorem}{Theorem}

\newtheorem{coro}[theorem]{Corollary}

\makeatletter
\def\ps@pprintTitle{%
 \let\@oddhead\@empty
 \let\@evenhead\@empty
 \def\@oddfoot{}%
 \let\@evenfoot\@oddfoot}
\makeatother

\begin{document}

\title{Controlling dynamics of a COVID--19 mathematical model using a parameter switching algorithm}
\vspace{5mm}

\author [rm1,rm2]{Marius-F. Danca}

\address[rm1]{Romanian Institute of Science and Tecçhnology, Cluj-Napoca, Romania,}
\address[rm2]{ STAR-UBB Institute, Babes-Bolyai University, Cluj-Napoca, Romania,}

\begin{abstract}
In this paper the dynamics of an autonomous mathematical models of COVID-19 depending on a real parameter bifurcation, is controlled by switching periodically the parameter value. For this purpose the Parameter Switching (PS) algorithm is used. With this technique, it is proved that every attractor of the considered system can be numerically approximated and, therefore, the system can be determined to evolve along, e.g., a stable periodic motion or a chaotic attractor. In this way, the algorithm can be considered as a chaos control or anticontrol (chaoticization)- like algorithm. Contrarily to existing chaos control techniques which generate modified attractors, the obtained attractors with the PS algorithm belong to the set system attractors. It is analytically shown that using the PS algorithm, every system attractor can be expressed as a convex combination of some existing attractors. Moreover, is proved that the PS algorithm can be viewed as a generalization of Parrondo's paradox.

\end{abstract}


\maketitle




\section {Introduction}
For a general nonlinear dynamical system, it is almost impossible to determine analytically its attractors. Therefore, generally the studies on nonlinear dynamics including invariant manifolds, attraction basins, heteroclinic and homoclinic orbits, Smale horseshoes, chaotic attractors rely on numerical analysis.
Under a variety of Lipschitz conditions, numerical methods for continuous-time dynamical systems, such as Runge-Kutta methods, define discrete dynamical systems. The comparison of the asymptotic behavior of the underlying dynamical system with the asymptotic behavior of its numerical discretization obtained with a convergent numerical scheme for ODEs is compared in \cite{bu,bu2}.
The obtained numerical approximations represent an important and natural part of a systematic analysis. Thus, if one considers an attractor $A$ of the considered dynamical system, the discrete dynamical system generated by a convergent numerical method can also have an attractor that converges to $A$ \cite{lars} (see also \cite{nou}).

As an important and natural part of a systematic analysis, numerical approximations of attractors are considered in this paper. Therefore, the notion attractor is understood here as being the \emph{numerical attractor}, obtained with some convergent numerical method (see e.g. \cite{atr,bu}), after transients are neglected.

Beside classical studies of mathematical models used to explain disease processes such as in \cite{biblio3,biblio1}, after the arrival of COVID-19 at the end of December 2019, several works on this pandemic have been published (see e.g. \cite{biblio4,biblio5,biblio6} and references therein).

In the context of the harms produced by the COVID-19, increased the interest of assessing the epidemic tendency.
Based on two official data sets, the National Health Commission of the People's Republic of China \cite{china} and the
Johns Hudson University \cite{jon}, a novel mathematical model of COVID-19 pandemic is proposed by Mangiarotti et al. in \cite{jon} is considered in this paper (see also \cite{gras}, where a fractional-order variant of the model is analyzed).

A variant of the model presented in \cite{jon} is described by the following Initial Value problem (IVP)
\begin{equation}\label{un1}
\begin{array}
[c]{cl}%
\overset{\cdot}{x}_{1}(t)= & -0.1053x_3^2(t) + 2.3430\times10^{-5} x_1^2(t) + 0.1521x_2(t)x_3(t)-0.0018x_1(t)x_2(t),\\
\overset{\cdot}{x}_{2}(t)= & 0.1606x_3^2(t)+\textbf{0.4404071}x_2(t)-\textbf{0.205\textbf{}2}x_1(t),\\
\overset{\cdot}{x}_{3}(t)= & \textbf{0.2845}x_3(t)-0.0001x_1(t)x_3(t)-1.2155\times10^{-5}x_1(t)x_2(t)+ 2.3788\times10^{-6}x_1^2(t),\\&x_1(0)=x_{10},~x_2(0)=x_{20}, ~x_3(0)=x_{30},
\end{array},
\end{equation}
where $x_1$ represents the daily number of new cases, $x_3$ represents the daily number of new deaths and $x_2$ represents the daily additional severe cases and $p$ is a real positive parameter. The significance of the bold coefficients is explained in Section \ref{sec1}.

The analysis presented in \cite{jon} show that the global modelling approach, could be useful for
decision makers to monitor the efficiency of control measures and to foresee the extent of the outbreak at various scales and also shows that the system could be used to adapt more classical modelling approaches to ensure mitigation and, hopefully, eradication of the disease.

In this paper we show both analytically and numerically that, via the Parameter Switching (PS) algorithm introduced in \cite{danprima}, the COVID-19 outbreak behavior can be controlled by switching periodically some of the system parameters.
In \cite{danprima} it is shown numerically that many of known systems can be controlled in the following sense: switching $p$ within a given set of values, in some deterministic manner (or even random \cite{danrand}) while the underlying Initial Value Problem (IVP) is numerically integrated, one can approximate some desired attractor with sufficiently small error. The convergence of the PS method is proved in \cite{dan2} (see \cite{dan3,dan1} for other proof approaches). In \cite{danx1} is analytically and also numerically proved the possibility to express any numerical attractor of a given dynamical system as a convex combination of some other existing attractors via the PS algorithm. The algorithm can be used both for theoretical studies of dynamical systems modeled such as synchronization \cite{dan7}, chaos control and anticontrol, or as generalization of the Parrondo paradox (see e.g. \cite{dan6,dan8}), or experimentally as well, like the implementation on real systems, e.g., electronic circuits \cite{dan4}.

In 1998 at the receipt of the Steele Prize for Seminal Contributions to Research, Zeilberger said that ``combining different and sometimes opposite approaches and view-points will lead to revolutions. So the moral is: Don\textsc{\char13}t look down on any activity as inferior, because two ugly parents can have beautiful children''. This paradox can be symbolically written as
``losing + losing = winning'' or, moreover, ``chaos+chaos=order'' or ''order+order=chaos''. The PS algorithm allows to implement this paradoxical game to a large class of dynamical systems which includes Lorenz system, Chen system, Roessler system, etc., to obtain chaos control-like or anticontrol-like (chaoticization). For Parrondo's paradox see \cite{par1,par2,par3}, few of the numerous works and therein references.

In this paper using the PS algorithm, it is shown that unwanted chaotic or regular behaviors of the system \eqref{un1} can be controlled in the sense that the system can be determined to evolve along of some desired regular or chaotic trajectory, respectively. Moreover, based on the attractor decomposition result in \cite{danx1}, every attractor of the system \eqref{un1} can be decomposed as function of other attractors.

\section{Parameter Switching algorithm}

In this section, the main properties of the PS algorithm are shortly presented (details can be found e.g. in \cite{danx1}).

Many single-parameter chaotic dynamical systems, such as the Lorenz system, R\"{o}ssler system, Chen system, Lotka--Volterra system, Hindmarsh-Rose system, etc. can be modeled as the following IVP
\begin{equation}\label{e0}
\dot x(t)=f(x(t)):=g(x(t))+pBx(t),\quad x(0)=x_0,
\end{equation}
\noindent where $t\in I=[0,T]$, $x=(x_1,x_2,...,x_n)^t$, $x_0\in \mathbb{R}^n$, $p\in \mathbb{R}$ is the bifurcation parameter, $B\in \mathbb{R}^{n\times n}$ a constant matrix, and $g : \mathbb{R}^n \rightarrow\mathbb{R}^n$ a continuous nonlinear function.

Because of the autonomous nature of systems \eqref{e0}, for simplicity, hereafter the time variable $t$ will be dropped.

An example of dynamical systems modeled by the IVP \eqref{e0} is the Lorenz system
\begin{equation*}
\begin{array}
[c]{cl}%
\overset{\cdot}{x}_{1}= & \sigma(x_{2}-x_{1}),\\
\overset{\cdot}{x}_{2}= & x_{1}(\rho-x_{3})-x_{2},\\
\overset{\cdot}{x}_{3}= & x_{1}x_{2}-\beta x_{3},
\end{array}
\label{lorenz}%
\end{equation*}
\noindent where $n=3$ with $a=10$ and $c=8/3$, if one considers $p=\rho$ (parameters $\sigma$ and $\beta$ can also be $p$) then
\[
g(x)=\left(
\begin{array}
[c]{c}%
\sigma(x_{2}-x_{1})\\
-x_{1}x_{3}-x_{2}\\
x_{1}x_{2}-\beta x_{3}%
\end{array}
\right)  ,~~B=\left(
\begin{array}
[c]{ccc}%
0 & 0 & 0\\
1 & 0 & 0\\
0 & 0 & 0
\end{array}
\right).
\]

Consider the following assumptions

\noindent
\begin{enumerate}\label {asu1}
\item [\textbf{H1}] Function $f$ in \eqref{e0} is Lipschitz continuous.
\end{enumerate}

\noindent
\begin{enumerate}\label {asu2}
\item [\textbf{H2}] To integrate the system \eqref{e0}, an explicit single $h$ step-size convergent numerical method is used.
\end{enumerate}

Under \textbf{H1}, with an admissible initial condition $x_0$ for any $p$, the IVP \eqref{e0} admits a unique and bounded solution.

Using methods assumed in \textbf{H2} is necessary only to explain the evolution of the PS algorithm. In this paper the utilized numerical method is the Standard Runge-Kutta (RK4).

Consider a dynamical system be modeled by the IVP \eqref{e0}. Denote a set of $N>1$ parameter values of $p$ by $\mathcal{P}_N=\{p_1,p_2,...,p_N\}$, $p_i\in \mathbb{R}$, $i=1,2,...,N$.
 The numerical method is used for solving the IVP \eqref{e0} on the discrete time nodes $nh$, $n=1,2,...$ of the discretized interval $I$.

Because of the solution uniqueness ensured by the Lipschitz continuity, one can consider that to each parameter $p_i\in\mathcal{P}_N$, $i\in\{1,2,...,N\}$, there corresponds a unique attractor $A_i$. Denote $\mathcal{A}_N=\{A_1,A_2,...,A_N\}$ the set of underlying attractors. Also, consider the set $\mathcal{P}_N$ ordered: $p_1<p_2<...<p_N$ \cite{danx1} (Fig.\ref{fig1}).

With above ingredients by switching $p$ within the set $\mathcal{P}_N$ according to a certain periodic rule while the IVP \eqref{e0} is numerically integrated with the RK4 method, the PS algorithm allows the approximation of any attractor of system \eqref{e0}.

Suppose one intends to generate some attractor, denoted $A^o$, corresponding to the value $p:=p^0$ but which, by some objective reasons, cannot be generated by integrating the IVP with this value of $p$ (situation often encountered in real dynamical systems).

The first step is to choose a set $\mathcal{P}_N$ such that the ends of the ordered set, $p_1$ and $p_N$, verify the relation (see Fig. \ref{fig1})

\begin{equation}\label{ordine}
p_1<p^0<p_N.
\end{equation}

For a given step-size $h>0$, the PS algorithm can be symbolized with the following scheme:

\begin{equation}\label{s0}
S:=[m_1\circ p_1,m_2\circ p_2,...,m_N\circ p_N]_h,
\end{equation}

\noindent where $\mathcal{M}_N=\{m_1,m_2,...,m_N\}$, $m_i\in \mathbb{N}^*$, $i=1,2,...,N$, denotes the ``weights'' of $p_i$ values.
The term $m_i\circ p_i$ indicates the number of $m_i$ times for which the parameter $p$ will take the value $p_i$, for $i\in\{1,2,...,N\}$ while the IVP is integrated.
Therefore, the scheme \eqref{s0} reads as follows: while the IVP \eqref{e0} with initial condition $x_0$ is numerically integrated with the fixed step-size method, for the first $m_1$ integration steps, at the nodes $nh$, $n=1,2,...,m_1$, $p$ will take the value $p_1$; for the next $m_2$ steps (at the nodes $nh$ with $n=m_1+1,m_1+2,...,m_2$), $p=p_2$; and so on, till the last $m_N$ steps, where $p=p_N$. Next, the algorithm repeats and begins with $p=p_1$ for $m_1$ times, and so on until the entire considered time integration interval $I$ is covered by the numerical integration. Therefore, the switching period of $p$, which is piece-wise constant function, is $\sum_{i=1}^Nm_ih$.

\vspace{3mm}
\noindent For simplicity, hereafter the index $h$ in \eqref{s0} is dropped.
\vspace{3mm}

For example the scheme $S=[3\circ p_1,2\circ p_2]$ indicates that the PS algorithm acts as follows: the integration over the first 3 consecutive steps, $p$ will take the value $p_1$, next, for the 2 consecutive steps $p$ will take the value $p_2$ after which the processus repeats.

Denote the solution obtained with the PS method, starting from the initial condition $y_0$, by $y_n$, and call it the \emph{switched solution}, and the solution, $x_n$, from the initial condition $x_0$, obtained by integrating the IVP \eqref{e0} with $p:=p^0$, where \cite{danx1}
\begin{equation}\label{p}
p^0=\frac{\sum_{i=1}^Nm_ip_i}{\sum_{i=1}^Nm_i},
\end{equation}
the \emph{averaged solution}. Also, the attractor corresponding to $p^0$, denoted $A^0$, is called the \emph{averaged attractor}, while the attractor obtained with the PS algorithm, denoted $A^*$, is called the \emph{switched attractor}.

In \cite{dan2,dan1,dan3} is proved that the attractor $A^0$ is approximated by the switched attractor $A^*$ generated by the PS method, the approximation being denoted $A^*\approx A^0$ and in \cite{danprima} the match between $A^*$ and $A^0$ is verified numerically by time series, histograms, Poincar\'{e} sections and Haussdorff distance.

\begin{remark}\label{remus1}
For a given $N$, the scheme \eqref{s0} is usually not unique: there are several sets $\mathcal{M}_N$ and $\mathcal{P}_N$ which generate the same value of $p^0$ via formula \eqref{p}.
\end{remark}

\begin{coro}\label{thx}
~\vspace{-.8em}
~\begin{itemize}
\itemsep-.3em
\item[i)] For each given sets $\mathcal{P}_N$ and $\mathcal{M}_N$, $A^*\approx A^0$, with $p^0$ given by \eqref{p};
\item[ii)] For each attractor $A$ of the considered system \eqref{e0}, there exists the set $\mathcal{P}_N$, such that $p_1<p^0<p_N$, and the set of weights $\mathcal{M}_N$, $N>1$, such that $A$ can be approximated by the PS method.
\end{itemize}
\end{coro}
\vspace{-1.2em}
\begin{proof} See Appendix A for a sketch of the proof.
\end{proof}
Denote next

\begin{equation}\label{al}
\alpha_j:=m_j/\sum_{i=1}^N m_i,~~ j=1,2,...,N.
\end{equation}
Since $\sum_{i=1}^N\alpha_i=1$, it is easy to see that $p^0$ is a convex combination of $p_i$, $i=1,2,3,...,N$, $p^0=\sum_{i=1}^N\alpha_i p_i$.

In \cite{danx1} it is proved that the set $ \mathcal{A}$, can be endowed with two binary relations (operators) $(\mathcal{A}, \oplus, \otimes )$, with $\oplus$ being \emph{addition of attractors} and $\otimes$ being \emph{multiplication of attractors by positive real numbers}. In this way,
 the following result presents a new modality to describe the averaged attractor $A^0$ as a convex-like combination of the attractors $A_i$, $i=1,2,3,...,N$.

\begin{coro}\label{propos}
For given sets $\mathcal{P}_N$ and $\mathcal{M}_N$, the average attractor $A^0$, corresponding to $p^0$ given by \eqref{p}, can be expressed as
\begin{equation}\label{a0}
A^0=\alpha_1\otimes A_1\oplus\ldots\oplus\alpha_N\otimes A_N.
\end{equation}
\end{coro}
\begin{proof}
See the sketch of the proof in Appendix B.
\end{proof}

\section{Control and anticontrol of the COVID-19 system \eqref{e0} with PS algorithm}\label{sec1}
If any of the bold constants in the system \eqref{un1} is consider as being the bifurcation parameter $p$, the system belongs to the class of systems \eqref{e0}. Consider one of the three possible choices of $p$ as coefficient of the variable $x_2$ in the second equation

\begin{equation}\label{asta}
\begin{array}
[c]{cl}%
\overset{\cdot}{x}_{1}= & -0.1053x_3^2 + 2.3430\times10^{-5} x_1^2 + 0.1521x_2x_3-0.0018x_1x_2,\\
\overset{\cdot}{x}_{2}= & 0.1606x_3^2+px_2-0.205x_1,\\
\overset{\cdot}{x}_{3}= & 0.2845x_3-0.0001x_1x_3-1.2155\times10^{-5}x_1x_2+ 2.3788\times10^{-6}x_1^2.
\end{array}.
\end{equation}

For simplicity, in Fig. \ref{fig2} is presented the bifurcation diagram of the first variable $x_1$ versus $p$, wherefrom one can see that the system presents a rich behavior.

 Note that due to the discrepancy between the very large values of the variables of $10^3$ order on one side and the very small values of the coefficients of order of about $10^{-5}$ on the other side, the numerical integrators used for the system \eqref{asta}, encounter difficulty in giving accurate results (see e.g. encircled parts in Fig. \ref{fig2} (a) and Fig. \ref{fig6} (a)).

As can be seen, the system \eqref{asta} belongs to the class of systems modeled by the IVP \eqref{e0} with

\[
g(x)=\left(
\begin{array}
[l]{l}%
-0.1053x_3^2 + 2.3430\times10^{-5} x_1^2 + 0.1521x_2x_3-0.0018x_1x_2\\
0.1606x_3^2-0.205x_1\\
0.2845x_3-0.0001x_1x_3-1.2155\times10^{-5}x_1x_2+ 2.3788\times10^{-6}x_1^2%
\end{array}
\right)  ,
\]
and
~~
\[B=\left(
\begin{array}
[c]{ccc}%
0 & 0 & 0\\
0 & 1 & 0\\
0 & 0 & 0
\end{array}
\right).
\]
Consider next some of the most representative cases.

\emph{Examples}
\begin{enumerate}
\itemsep-.3em

\item \label{exx1}Consider first one intends to force the system to evolve along the stable cycle corresponding to the parameter value $p=0.423$ with the PS algorithm (see the zoomed image in Fig. \ref{fig1} b)), i.e. to approximate $A^0$, with $p_0=0.423$, with the attractor $A^*$ obtained with the PS algorithm. For this purpose, one can find two other values $p_1$ and $p_2$ such that $p_1<0.423<p_2$ (see relation \eqref{ordine}). Let the values of $p_{1,2}$ and related weights $m_{1,2}$ such that the relation \eqref{p} gives $p^0=0.423$. On of the simplest choices is $p_1=0.422$, $p_2=0.424$ and $m_1=m_2=1$ (see Fig.\ref{fig1} b)). Therefore: $p^0=(m_1p_1+m_2p_2)/(m_1+m_2)=0.423$ and the PS algorithm will act with the switching scheme $[1\circ p_1,1\circ p_2]$. The obtained switched attractor $A^*$ approximates the averaged attractor $A^0$, corresponding to $p^0=0.423$, $A^*\approx A^0$, as can be seen in Fig. \ref{fig3} a) where, both attractors are overplotted after transients are discarded (red and blue, respectively). The match between the two attractors is also verified by Poincar\'{e} section with the plane $x_3=80$ (Fig. \ref{fig3} b) and histograms (Figs. \ref{fig3} c) and d) respectively).

\item \label{exx2}Another stable periodic motion, corresponding to $p=0.43$, can be obtained with the PS algorithm via the scheme $[1\circ p_1,1\circ p_2, 3\circ p_3]$ with $p_1=0.4265$, $p_2=0.4287$ and $p_3=0.4316$ (Fig. \ref{fig2} b) and $m_1=m_2=1$ and $m_3=3$. As the bifurcation diagram shows, around the value $p=0.43$, the window around this parameter value contains merged very interleaved thin periodic and chaotic windows and, therefore, the considered cycle is difficult to approximate.
However, with acceptable error, the switched attractor $A^*$ (red plot in Figs. \ref{fig4}) approximates the averaged attractor corresponding to $p=0.43$ (blue plot Figs. \ref{fig4}). Figs. \ref{fig4} a-d) show the phase plot of the two attractors, their Poincar\'{e} section with the plane $x_3=80$, and histograms, respectively. The mentioned relative small error can be remarked in the Poincar\'{e} section.
\item \label{exx3}
Not only stable periodic trajectories can be approximated by the PS algorithm, but also chaotic trajectory. Thus, suppose one intends to force the system to evolve along the chaotic trajectory corresponding to $p=0.42905$ (Figs. \ref{fig2} (b)) considering e.g. the values $p_1=0.4265$ and $p_2=0.4316$, with the scheme $[1\circ p_1,1\circ p_2]$. The result is presented in Figs. \ref{fig5}. As expected, and as shown by the phase plot (Fig. \ref{fig5} (a)), Poincar\'{e} section (Fig. \ref{fig5} (b)) and histograms (Figs. \ref{fig5} (c), (d), respectively), due to the finite time in which the PS acts, the match between the two chaotic attractors is only an asymptotical process.
 \item\label{exx4} While in the above case two chaotic attractors have been used to generate another chaotic attractor, one can approximate a chaotic attractor, e.g. the one corresponding to $p^0=0.4286$, using two values, e.g., $\mathcal{P}_2=\{0.423,0.43\}$ corresponding to stable cycles, and $\mathcal{M}_2=\{1,4\}$ for which, via \eqref{p}, $p^0=0.4286$ (see Fig. \ref{fig2} (b)). Similarly, one can approximate some stable cycle starting from two other stable cycles. For example one can obtain the stable cycle corresponding to $p=0.423$, starting from two (or more) values framing the value $0.423$.

\end{enumerate}

As mentioned at the beginning of this section, the form of the system \eqref{e0}, allows other two choices of the bifurcation parameter $p$. One of them is

\begin{equation}\label{as}
\begin{array}
[c]{cl}%
\overset{\cdot}{x}_{1}= & -0.1053x_3^2 + 2.3430\times10^{-5} x_1^2 + 0.1521x_2x_3-0.0018x_1x_2\\
\overset{\cdot}{x}_{2}= & 0.1606x_3^2+0.4404071x_2-px_1\\
\overset{\cdot}{x}_{3}= & 0.2845x_3-0.0001x_1x_3-1.2155\times10^{-5}x_1x_2+ 2.3788\times10^{-6}x_1^2,
\end{array}.
\end{equation}
with the bifurcation diagram presented in Fig. \ref{fig6}. As shown in the case of the system \eqref{asta}, by using the PS algorithm, the system \eqref{as} can be determined to evolve, e.g., along the stable cycle corresponding to $p=0.2115$ with, e.g., the scheme $[1\circ 0.211,5\circ 0.2116]$ (Fig. \ref{fig6} (b))

\begin{remark}
i) As happens often in Nature, there are many systems where, accidentally or not, one or several parameters switch less or more periodically their values such that the system changes his dynamics. As proved in \cite{danrand} the PS algorithm can be applied even randomly in the sense that given a set of parameters, $\mathcal{P}_N$, with underlying weights $\mathcal{A}_N$, changing randomly the order of the parameters, the approximation still works.

ii) The bifurcation diagram is useful to understand the way in which the PS algorithm works and also to allow easily the choice of parameters. However, as Corollary \ref{thx} shows, without a bifurcation diagram by choosing some set $\mathcal{P}_N$ with some weights $\mathcal{   M}_N$, the PS algorithm always approximates an attractor $A^0$, with $p^0$ given by \eqref{p}.
\end{remark}

\section{Attractors decomposition and Parrondo's paradox}

Denote for clarity the attractor corresponding to some parameter value $p$ with $A_p$.

Consider the attractors $A_i$, $i=1,2,...,N$ as being chaotic with behavior denoted by $chaos_i$, $i=1,2,...,N$ and $A^0$ a regular attractor, whose behavior is denoted $order$. Then, \eqref{a0} can be written symbolically as follows

\begin{equation}\label{par}
order=chaos_1+chaos_2+...+chaos_N,
\end{equation}
i.e., a generalized form of the Parrondo paradox, where only two participants are considered. The coefficients $\alpha_i$ maintain their weight role of each dynamic.

As seen before, the attractor corresponding to $p=0.423$, denoted $A_{0.423}$ (Example \ref{exx1}), has been approximated by the switched attractor obtained with the PS via the scheme $[1\circ p_1,1\circ p_2]$.
Since, conform \eqref{al}, $\alpha_1=\alpha_2=\frac{1}{2}$, and following the decomposition relation \eqref{a0}, the attractor $A_{0.423}$ can be decomposed as

\begin{equation}\label{ex1}
A_{0.423}:=A^0=\frac{1}{2}\otimes A_{0.422}\oplus\frac{1}{2}\otimes A_{0.424}.
\end{equation}
Because the two attractors used in the PS algorithm are chaotic, if one denote the chaotic behaviors by $chaos_1$ and $chaos_2$, respectively, and the obtained regular motion by $order$, the relation \eqref{ex1} can be written in parrondian terms as follows

\[
chaos_1+chaos_2=order.
\]
Therefore, in this case the PS algorithm acts as a control-like algorithm.

If one denote $\alpha_1=\alpha_2=\frac{1}{5}$, and $\alpha_3=\frac{3}{5}$ the attractor corresponding to $p=0.43$ obtained in Example \ref{exx2}, can be expressed as follows

\begin{equation}\label{ex2}
A_{0.43}=\frac{1}{5}\otimes A_{0.4265}\oplus\frac{1}{5}A_{0.4287}\oplus\frac{3}{5}A_{0.4316},
\end{equation}
which, again, can be in terms of Parrondo's paradox as the following chaos control-like

\[
chaos_1+chaos_2+chaos_3=order,
\]
where $chaos_i$, $i=1,2,3$ represent the chaotic behavior of the three used attractors in the PS algorithm.

In the case of Example \ref{exx3}, the obtained attractor $A_{0.42905}$ can be decomposed as

\begin{equation}\label{ex3}
A_{0.42905}=\frac{1}{2}\otimes A_{0.4265}\oplus\frac{1}{2} \otimes A_{0.4316},
\end{equation}
relation which can be written as
\[
chaos_1+chaos_2=chaos_3,
\]
which do not more represent a Parrondo paradox.
\noindent Considering the Example \ref{exx4}, the chaotic attractor $A_{0.4286}$, which can be decomposed as

\[
A_{0.4286}=\frac{1}{5}\otimes A_{0.423}\oplus\frac{4}{5}\otimes A_{0.43},
\]
can be expressed as an anticontrol-like, in which switching the parameter within a set of values which generate stable motion, the obtained attractor is a chaotic one
\[
order_1+order_2=chaos.
\]

Because of the PS algorithm properties, every stable or chaotic attractor can be obtained by chaos control or anticontrol-like, using the only condition that the considered value $p^0$ is framed (see \eqref{ordine}) by values $p_1$ and $p_N$ which are of opposite kind (generate chaotic behavior in the case of chaos control-like, or regular behavior in the case of anticontrol-like).

There could imagine two main situations when the algorithm can be considered as chaos control or anticontrol. The most important is the chaos control. Suppose the system evolves chaotically for some parameter value $p_1$ and one intends to change this behavior and stabilize it such to evolve along a stable attractor corresponding to the $p^0$ value which, for some objective reasons, cannot be set. Then, chosen another admissible value $p_2$, for which the system evolve chaotic, such that \eqref{p} gives for adequate weights $m_{1,2}$ the value $p^0$, the PS algorithm approximates the desired stable attractor $A^0$. Obviously, there could be chosen several parameters $p_i$, $i=1,2,...,N$ with corresponding chaotic (or regular) behaviors which generates the same attractor $A^0$ (see Remark \ref{remus1}).

\section{Conclusion and discussions}
In this paper it is shown how the PS algorithm can be used to obtain a stabilization of the chaotic behavior of a pandemic like COVID-19, modeled by the relation \eqref{e0}. In support of this idea note that the most mathematical models of integer or even fractional order describing COVID-19 can be described by \eqref{e0} as can be seen in e.g. few of the numerous works on this subject \cite{prima,doua,treia,patra, cincia,sasea,saptea,opta}.
Moreover, the algorithm proved to be useful experimentally too \cite{dan4}.
The method determines the system to change the behavior to any other of its possible regular or chaotic behavior. One of the main advantages of the PS algorithm is the fact it can be applied to most of the known dynamical systems. Also, compared to the classical methods of chaos control, where due to the way in which the parameter is changed, the obtained stable evolution has a new behavior, different to the potential system attractors, in the case of the PS algorithm, the obtained attractors belong to set of all admissible attractors. Moreover, the PS algorithm allows to generalize of the Parrondo's paradox.
Beside the possibility to approximate any desired behavior of a system modeled by \eqref{e0}, the PS algorithm provides a new and interesting possibility to express any attractor as a convex combination of other attractors, like in the considered COVID-19 system where, for example, a stable attractor can be considered as a convex combination of other, chaotic attractors.


\newpage{\pagestyle{empty}\cleardoublepage}

\begin{figure}
\begin{center}
\includegraphics[scale=0.9]{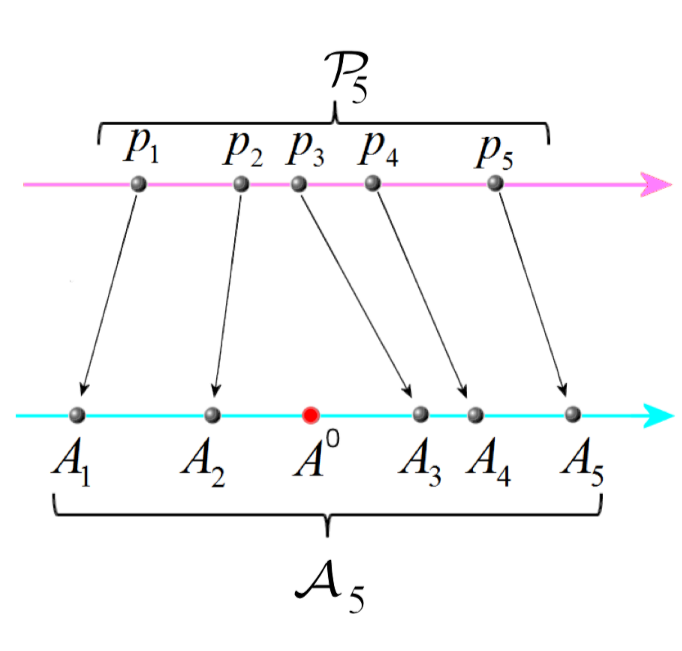}
\caption{Sketch of the sets $\mathcal{P}_5$ and $\mathcal{A}_5$ (after \cite{danx1}).}
\label{fig1}
\end{center}
\end{figure}

\begin{figure}
\begin{center}
\includegraphics[scale=0.85]{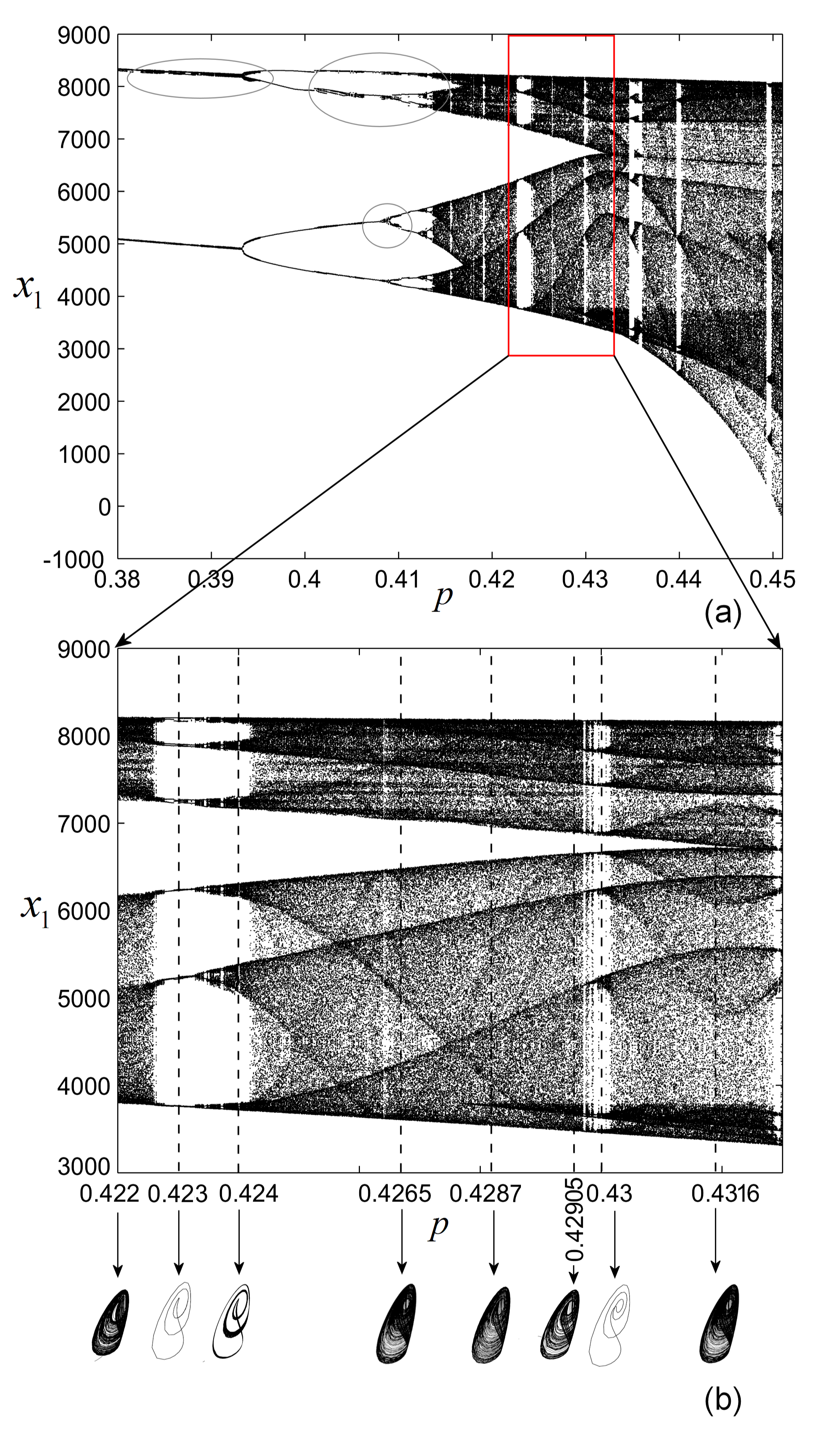}
\caption{(a) Bifurcation diagram of the first component $x_1$ of the COVID-19 system \eqref{un1}; (b) Zoomed image to unveil periodic windows.}
\label{fig2}
\end{center}
\end{figure}

\begin{figure}
\begin{center}
\includegraphics[scale=0.45]{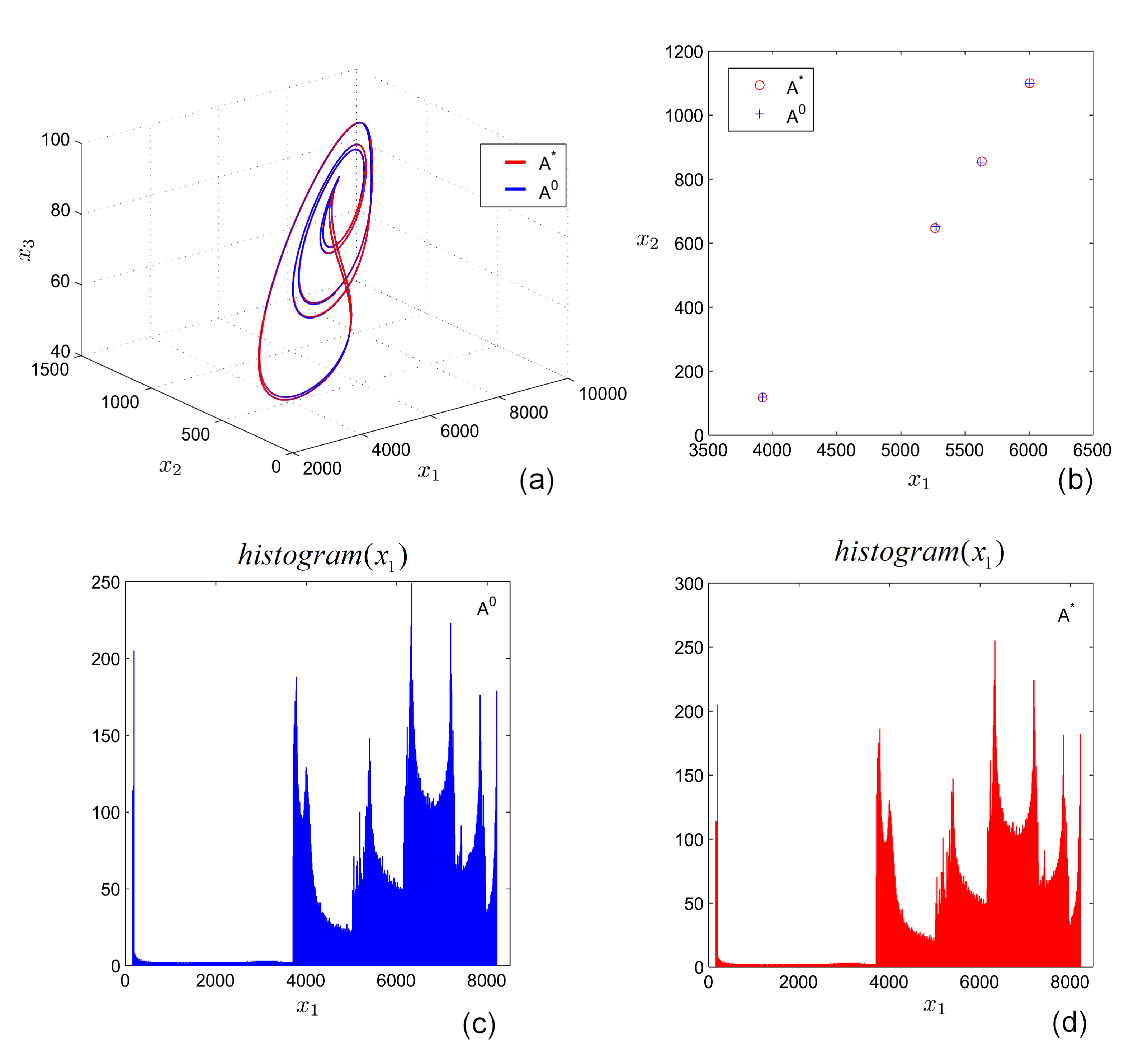}
\caption{Chaos control-like. Switched and averaged attractors $A^*$ and $A^0$ using the scheme $[1\circ p_1,1\circ p_2]$ with $p_1=0.422$ and $p_2=0.424$ (red and blue, respectively). (a) Phase plot; (b) Poincar\'{e} section with the plane $x_3=80$; (c)-(d) Histograms of the first component $x_1$ of the averaged and switched attractors, respectively. }
\label{fig3}
\end{center}
\end{figure}

\begin{figure}
\begin{center}
\includegraphics[scale=0.5]{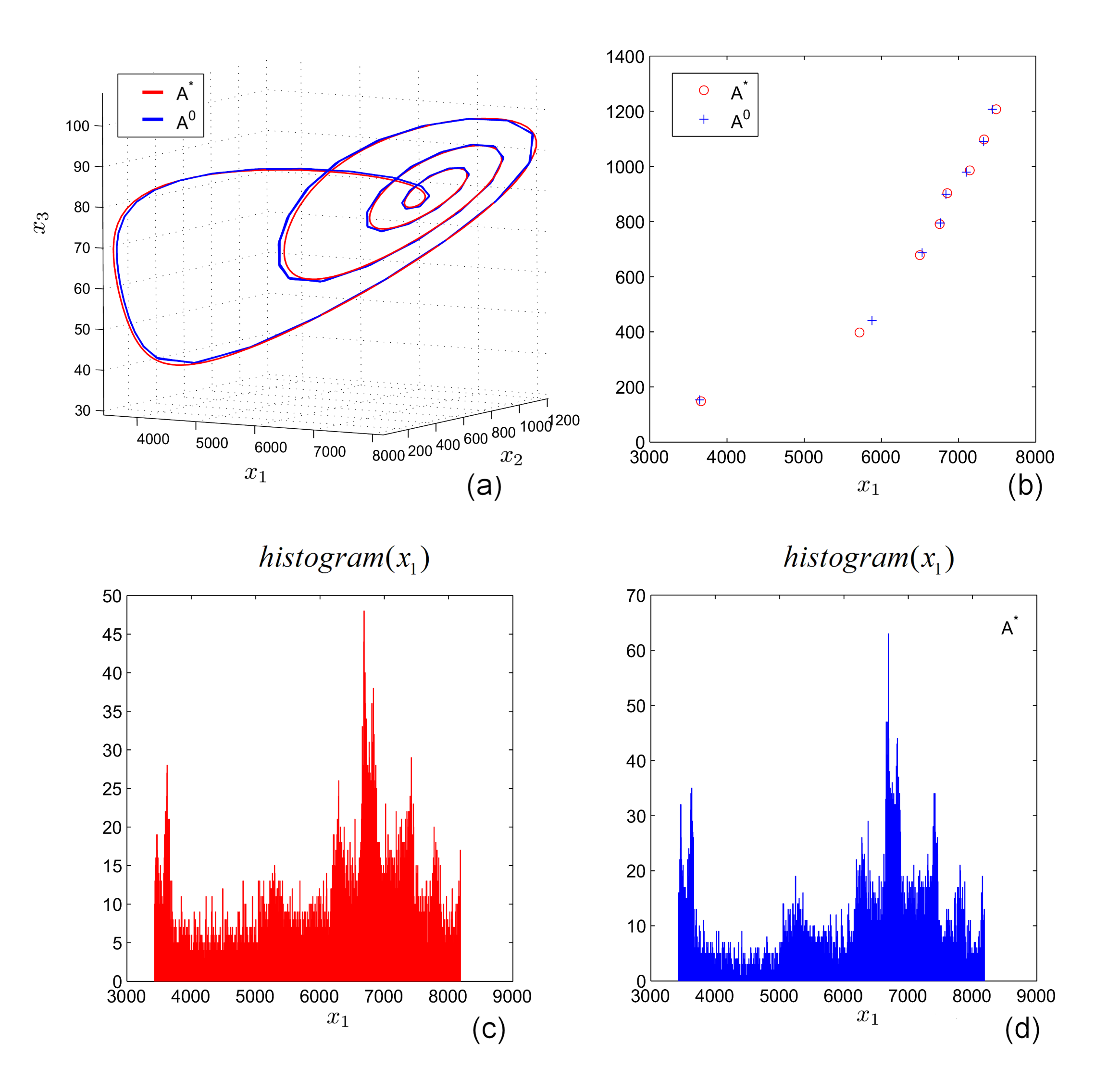}
\caption{Chaos control-like. Switched and averaged attractors $A*$ and $A^0$ using the scheme $[1\circ p_1,1\circ p_2,3\circ p_3]$ with $p_1=0.4265$, $p_2=0.4287$ and $p_3=0.4316$ (red and blue, respectively). (a) Phase plot; (b) Poincar\'{e} section with the plane $x_3=80$; (c)-(d) Histograms of the first component $x_1$ of the averaged and switched attractors, respectively. }
\label{fig4}
\end{center}
\end{figure}

\begin{figure}
\begin{center}
\includegraphics[scale=0.45]{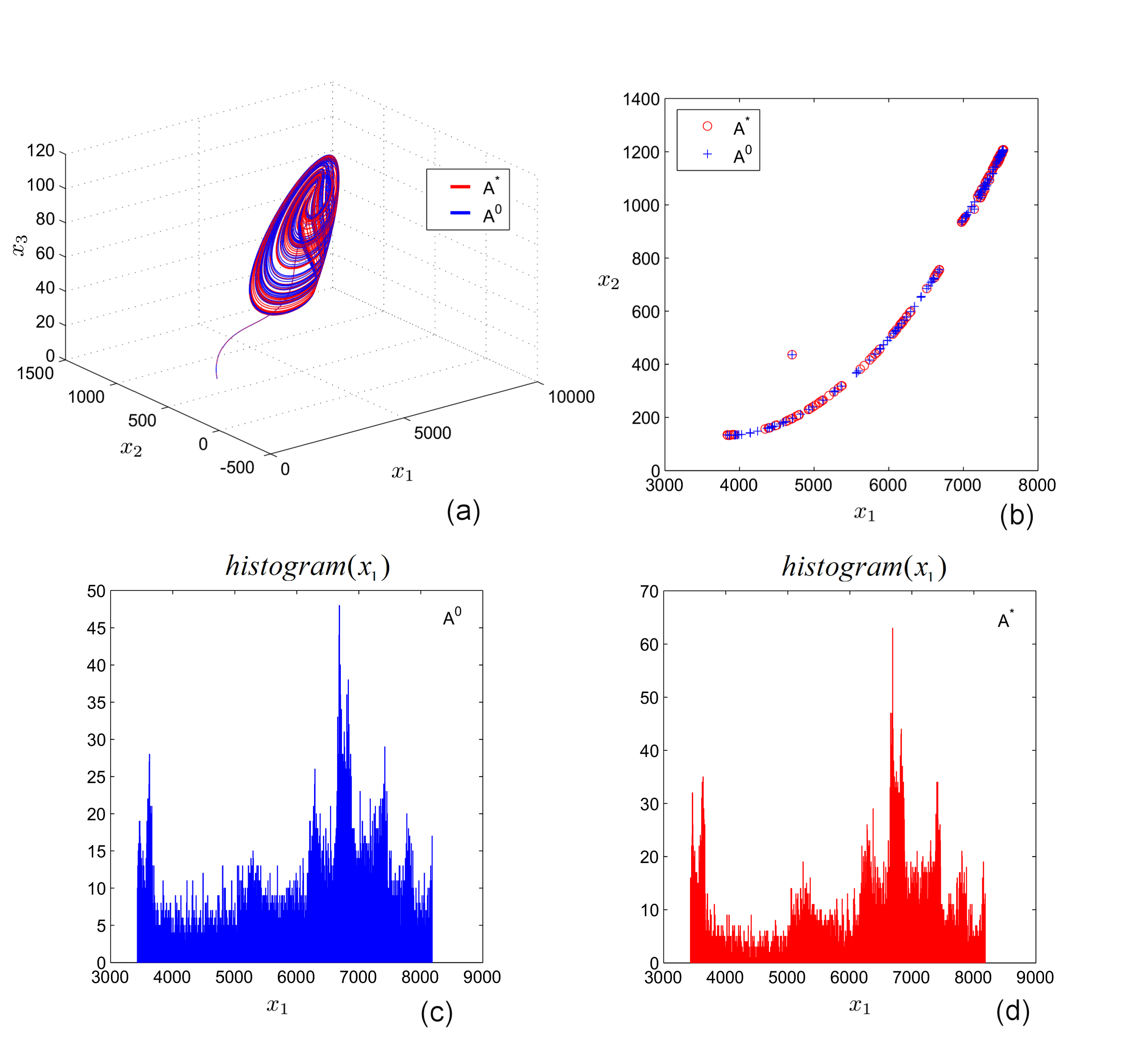}
\caption{Switched and averaged attractors $A*$ and $A^0$ using the scheme $[1\circ p_1,1\circ p_2]$ with $p_1=0.4265$ and $p_2=0.4316$ (red and blue, respectively). (a) Phase plot; (b) Poincar\'{e} section with the plane $x_3=80$; (c)-(d) Histograms of the first component $x_1$ of the averaged and switched attractors, respectively. }
\label{fig5}
\end{center}
\end{figure}

\begin{figure}
\begin{center}
\includegraphics[scale=0.65]{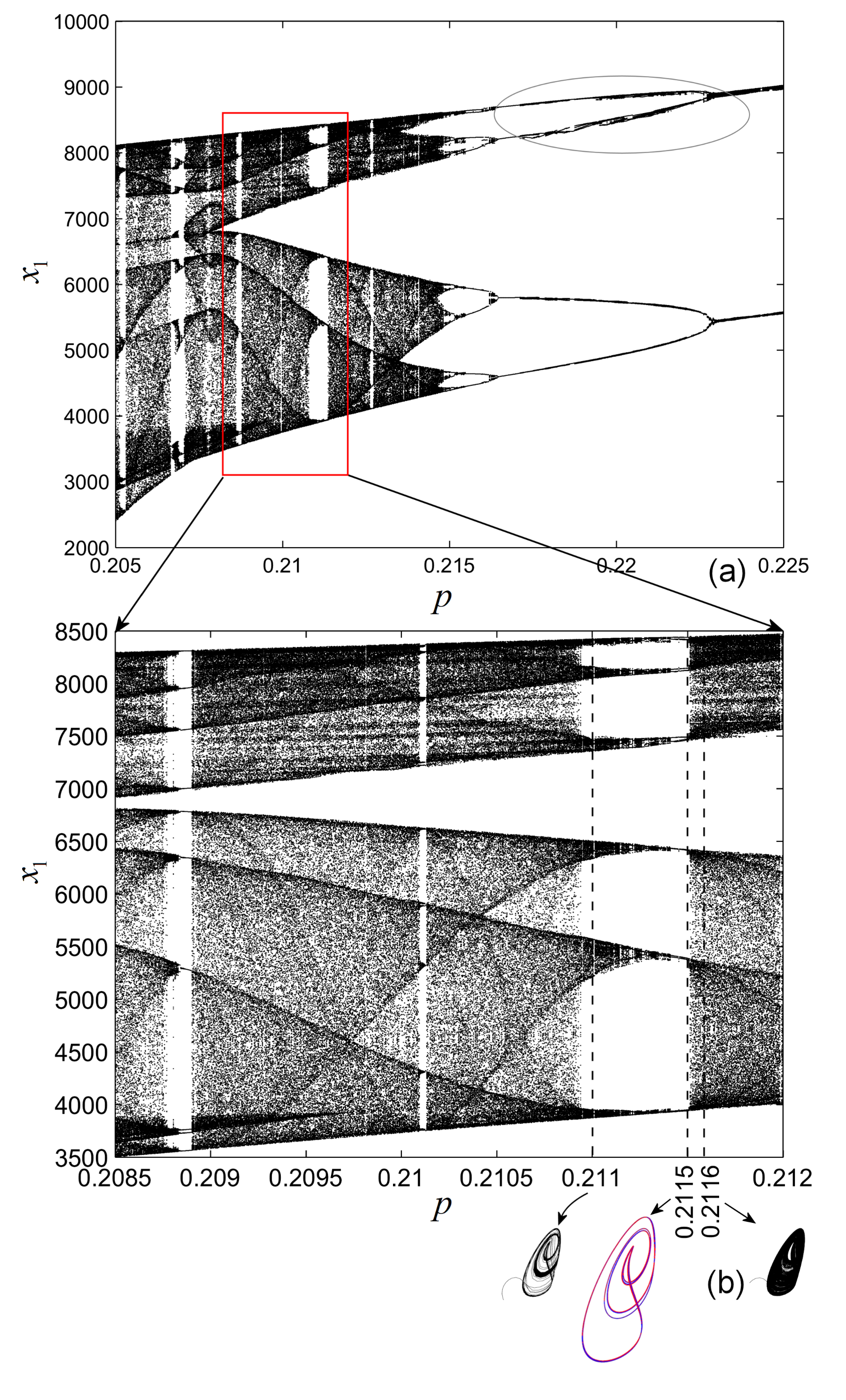}
\caption{(a) Bifurcation diagram of the first component $x_1$ of the COVID-19 system \eqref{as}; (b) Zoomed image.}
\label{fig6}
\end{center}
\end{figure}

\newpage{\pagestyle{empty}\cleardoublepage}

\appendix

\section{Sketch of the proof of Corollary 1}\label{apcoro1}
One of the existing proofs of the convergence of the PS algorithm shows that the switched solution approximates a solution of the linearized system of the averaged system \cite{dan1}.
Consider the \emph{switched system}

\begin{equation}\label{hhh}
\overset{\cdot}{y}(t)=g(y(t))+p(t/\lambda)By(t), ~ y(0)=y_0,
\end{equation}
where $p:I\rightarrow \mathbb{R}$ is a periodic function with period $T$, having the averaged value
\[
\frac{1}{T}\int_{t}^{t+T}p(u)du=q, ~~\forall t\in I.
\]
In terms of the scheme $S$ \eqref{s0}, $T=\sum_{i=1}^N m_ih$.
Also, consider the \emph{averaged system}
\begin{equation}\label{ppp}
\overset{\cdot}{x}(t)=g(x(t))+qBx(t), ~ x(0)=x_0.
\end{equation}
Let $s$ the unique solution of \eqref{ppp} (see Assumption \textbf{H1}) in whose neighborhood the linearization of \eqref{ppp} is
\[
\overset{\cdot}{\varepsilon}=[G(x(t))+qB]{\varepsilon}, ~~{\varepsilon}(0)={\varepsilon}_0,
\]
where ${\varepsilon}(t)=x(t)-s(t)$, and $G$ is the Jacobian of $g$ evaluated at $s$.

\noindent Next, by linearizing \eqref{hhh} with $e(t)=y(t)-s(t)$ one obtains
\[
\overset{\cdot}{e}(t)=[G(t)+p(t/\lambda)B]e(t), ~~ e(0)=e_0.
\]
Next it can be proved that there exists $\lambda>0$ such that $lim_{t\rightarrow \infty} ||e(t)-{\varepsilon}(t)||=\delta(\lambda^2)$, where $\delta(\lambda^2)$ is an order function. Here, $\lambda$ is used to set the length of the integration step size $h$.

\section{Sketch of the proof of Corollary 2}\label{apcoro2}

Under uniqueness assumption (\textbf{H1}), one can naturally assume that there exists a linear (bijective)
order-preserved mapping \cite{danx1} (see Fig. \ref{fig1})
\[
H:\mathcal{P}\rightarrow \mathcal{A},
\]
where $\mathcal{P}$ is the set of all admissible parameters value and $\mathcal{A}$ the set of corresponding attractors. A general way of defining $\oplus$ and $\otimes$ is

\[
\alpha\otimes A:=H(\alpha H^{-1}(A)),
\]
and
\[
A_1\oplus A_2:=H(H^{-1}(A_1)+H^{-1}(A_2)).
\]

Next, using the relations $H^{-1}(A_i)=p_i$, $H^{-1}(H(\alpha_ip_i))=\alpha_ip_i$, and the expression of $p^0$, given by \eqref{p} the relation \eqref{a0} can be easily verified.

\end{document}